\documentclass{jgcc}

\pdfoutput=1
\usepackage{lastpage}
\jgccdoi{12}{2}{1}{6200}  
\jgccheading{}{\pageref{LastPage}}{}{}%
{Mar.~15,~2020}{Jul.~13,~2020}{}

\keywords{finitely presented groups, random groups, small cancellation.}

\usepackage{graphicx,amsfonts,mathscinet,amssymb,mathtools}
\usepackage[caption=false]{subfig}

\usepackage{hyperref}

\DeclareMathOperator{\FR}{FR}
\DeclareMathOperator{\CR}{CR}

\DeclareMathOperator{\C}{C}
\DeclareMathOperator{\NC}{NC}

\begin{document}

\title{Density of Metric Small Cancellation in Finitely Presented Groups}

\author{Alex Bishop}
\address{University of Technology Sydney, Australia}
\urladdr{\url{https://alexbishop.github.io}}
\email{\href{mailto:alexbishop1234@gmail.com}{alexbishop1234@gmail.com}}

\author{Michal Ferov}
\address{University of Newcastle, Australia}
\email{\href{mailto:michal.ferov@gmail.com}{michal.ferov@gmail.com}}

\begin{abstract}
	Small cancellation groups form an interesting class with many desirable properties.
	It is a well-known fact that small cancellation groups are generic; however, all previously known results of their genericity are asymptotic and provide no information about ``small'' group presentations.
	In this note, we give closed-form formulas for both lower and upper bounds on the density of small cancellation presentations, and compare our results with experimental data.
\end{abstract}
\maketitle

\section{Introduction}

Informally speaking, a group is a $\C'(\lambda)$ small cancellation group if it is given by a presentation that satisfies the $\C'(\lambda)$ metric small cancellation condition, i.e.\ a presentation where no two relators share a common segment of proportion $\lambda$ (see Subsection~\ref{subsec:definition-of-small-cancellation} for a formal definition). Small cancellation groups form a class with many desirable algebraic and algorithmic properties. For example, both the word problem and the conjugacy problem are uniformly solvable in linear time by Dehn's algorithm~\cite{Greendlinger}. Following \cite[Lemma~3]{Arzhantseva} and \cite[Section~9.B]{Gromov_asymptotic} it is known that small cancellation presentations are ``generic'' meaning that a random presentation will most likely present a small cancellation group (see Subsection~\ref{subsec:random-groups}). Furthermore, given a finite presentation one can easily check whether or not it satisfies the small cancellation property: all one needs to do is to inspect all pairs of relators for a common segment of critical length. For these reasons small cancellation groups were suggested as a platform for computation in several cryptographic protocols (see~\cite{Zapata1, Zapata2, Cavallo, Habeeb}).

The results of \cite{Arzhantseva} and \cite{Gromov_asymptotic} on genericity of small cancellation groups are asymptotic, stating that a ``big enough presentation'' will, with overwhelming probability, be a small cancellation presentation. In particular, neither of these papers specify how big is ``big enough''. For practical applications, such as in cryptography, this is not sufficient. In this paper we improve the aforementioned results by giving closed-form formulas for both a lower and an upper bound on the probability that a random presentation satisfies the small cancellation condition.
Moreover, using these bounds, we are able to derive the asymptotic bounds on genericity as given in \cite{Arzhantseva,Gromov_asymptotic}.

In Lemma~\ref{sec:lower-bound}, we will see that we have a lower bound as follows.

\begin{thm}%
    \label{thm:lower}
    There is a function $p^{\leqslant}_\lambda(r, \ell_1, \ell_2, m)$ given by a closed-form formula such that a presentation chosen uniformly at random from the set of all presentations of the form $\langle X \mid W \rangle$, where $|X| = r$, $|W| = m$ and $\ell_1 \leqslant |w| \leqslant \ell_2$ for each $w \in W$, is power-free and satisfies the metric small cancellation condition $\C'(\lambda)$ with probability at least $p^{\leqslant}_\lambda(r, \ell_1, \ell_2, m)$.
    Moreover, we have
    \[
    	1 - p^\leqslant_\lambda(r,\ell_1,\ell_2,m)
    	\leqslant
    		8 m^2 r\ell_2^2 (\ell_2-\ell_1+1) (2r-1)^{-\lambda\ell_2-1},
    \]
    and thus $\lim_{\ell_2\to\infty} p^{\leqslant}_\lambda(r, \ell_1, \ell_2, m) = 1$ for each fixed $r \geqslant 2$, $\lambda$ and $m$.
\end{thm}

Moreover, from Propositions~\ref{prop:upper-bound-on-probability} and \ref{prop:necessary bound on m} we have an upper bound on the probability of small cancellation given in Theorem~\ref{thm:upper} below.

\begin{thm}%
    \label{thm:upper}
    There is a function $p^{\geqslant}_\lambda(r, \ell, m)$ given by a closed-form formula such that a presentation chosen uniformly at random from the set of all presentations of the form $\langle X \mid W \rangle$, where $|X| = r$, $|W| = m$ and $|w| = \ell$ for all $w \in W$, is power-free and satisfies the metric small cancellation condition $\C'(\lambda)$ with probability at most $p^{\geqslant}_\lambda(r, \ell, m)$.
    Moreover, for each $m \geqslant 1$ and $r\geqslant 2$, we have
    \[
    	\ln(1/p^{\geqslant}_\lambda(r, \ell, m))
    	\geqslant
    	\frac{1}{8} {(m-1)}^2 \ell {(2r-1)}^{-\lceil\lambda\ell\rceil},
    \]
    that is, $p^{\geqslant}_\lambda(r, \ell, m)$ is not simply the constant function $1$.
    Notice that from Theorem~\ref{thm:lower} we have $\lim_{\ell\to\infty} p^{\geqslant}_\lambda(r, \ell, m) = 1$, and thus
    \[
    	\lim_{\ell\to\infty}\ln(1/p^{\geqslant}_\lambda(r, \ell, m)) = 0
    \]
    for each fixed $r \geqslant 2$, $\lambda$ and $m$.
\end{thm}

Using the lower bound presented in Section~\ref{sec:lower-bound}, we show that the probability of obtaining a small cancellation presentation is non-trivial even for relatively small parameters, and compare our results with experimental data.

The organisation of the paper is as follows.
In Section~\ref{sec:preliminaries}, we provide the preliminary notions;
in particular, in Subsection~\ref{subsec:definition-of-small-cancellation} we recall the formal definition of metric small cancellation, and in Subsection~\ref{subsec:random-groups} we recall the notion of random groups.
In Section~\ref{sec:density-of-small-cancellation}, we give the main results of this paper; in particular, Subsection~\ref{sec:lower-bound} derives a lower bound for the probability of small cancellation in terms of the given parameters of the presentation, and in Subsection~\ref{sec:upper-bound} we give an upper bound.
In Section~\ref{sec:relator-length}, we combine these two bounds to discuss the limitations on the choice of parameters in regards to maximise the probability of small cancellation.
Finally, in Appendix~\ref{sec:experimental-results} we compare our theoretical results with experimental data; in particular, we provide several heat maps which show how our bounds differ as we vary the parameters of the presentation.

\section{Preliminaries}%
\label{sec:preliminaries}

Given a finite set $X = \{ x_1, x_2, \ldots, x_r \}$, we denote the free group generated by $X$ as $F(X)$.
Further, we write $w \in F(X)$ to denote that $w$ is a freely reduced word in $\left(X^{\pm 1}\right)^*$, that is, $w$ does not contain $xx^{-1}$ or $x^{-1}x$ as a factor for any $x \in X$.
Notice that each word $w \in F(X)$ corresponds to a unique element of the free group on the generating set $X$.

Let $w = x_{i_1}^{\epsilon_1} x_{i_2}^{\epsilon_2} \cdots x_{i_k}^{\epsilon_k}$ with $x_{i_1}, x_{i_2}, \ldots, x_{i_k} \in X$ and each $\epsilon_j \in \{-1,1\}$.
Then we define the \emph{word length} of $w$ as $|w|_X = k$; and $|w|$ when the generating set $X$ is clear from the context.
Further, for each $0 \leqslant d < k = |w|$ we write $w_{\ll d}$ to denote the \emph{(left) cyclic permutation of $w$ by a distance of $d$}, that is,
\[
	w_{\ll d}
	=
	x_{i_{d+1}}^{\epsilon_{d+1}}
	x_{i_{d+2}}^{\epsilon_{d+2}}
	\cdots
	x_{i_k}^{\epsilon_k}
	\ 
	x_{i_1}^{\epsilon_1}
	x_{i_2}^{\epsilon_2}
	\cdots
	x_{i_d}^{\epsilon_d}.
\]
We say that a word $w$ is \emph{cyclically reduced} if all of its cyclic permutations are freely reduced, or equivalently, if $w = x_{i_1}^{\epsilon_1} x_{i_2}^{\epsilon_2} \cdots x_{i_k}^{\epsilon_k}$ is freely reduced and $x_{i_1}^{\epsilon_1} \neq x_{i_k}^{-\epsilon_k}$ with respect to the free group $F(X)$.

Let $X$ be a set with $r = |X|$ elements, and $W \in F(X)^m$ be a list of $m$ words, where each $w \in W$ is cyclically reduced; then $\langle X \mid W \rangle$ is a presentation with $r$ generators and $m$ relators.
Notice that $W$ may contain the same element twice; and further presentations that differ only by permuting relators are considered to be distinct. For example $\langle x,y \mid x^2, y^2 \rangle$ and $\langle x,y \mid  y^2, x^2 \rangle$ are considered to be distinct presentations.
For the ease of writing, as a slight abuse of notation, $w \in W$ will denote that there is an $i \in \{1, 2, \ldots, m\}$ such that $\pi_i(W) = w$, where $\pi_i \colon F(X)^m \to F(X)$ is the projection onto the $i$-th component of $F(X)^m$.

\subsection{Small Cancellation Presentations}%
\label{subsec:definition-of-small-cancellation}

The notation and terminology used in this section follows that of \cite{LS}.

We denote the \emph{symmetric closure} of a finite list of words $W \subset F(X)^*$ as
\[
	W^S
	=
	\left\{
    	w_{\ll d}, (w_{\ll d})^{-1}
    \, \middle\vert \,
    	w \in W
    	\text{ and }
    	0 \leqslant d < |w|
	\right\}.
\]
We say that a word $u$ is a \emph{symmetric consequence} of a word $w$ if $u \in (w)^S$.
Further, we say that $W$ is \emph{minimal} if there is no proper sublist $U$ such that $U^S = W^S$.
For example, $(aaaa,baba,abab)$ is not minimal as $baba = abab_{\ll 1}$, however, the list $(aaaa,abab)$ is minimal.

Let $w \in F(X)$, then the maximum size of $(w)^S$ is given by $2|w|$, that is, $|(w)^S| \leqslant 2|w|$.
Notice that a cyclically reduced word $w$ factors as a proper power, $w = u^n$, with $n > 1$, if and only if $|(w)^S| < 2|w|$; thus we say that $w$ is \emph{power-free} if we have $|(w)^S| = 2|w|$.

Let $\mathcal{P} = \left\langle X \,\middle\vert\, W \right\rangle$ be a presentation where $W \in F(X)^*$ is the list of cyclically reduced relators.
Then, we say that $\mathcal{P}$ has \emph{metric small cancellation $\C'(\lambda)$} if the list $W$ is minimal, each word $w$ in the list $W$ is power-free, and any pair of words $u,w \in W^S$ may only have a short common prefix; in particular, if $u = pa$, $w = pb$ with $p,a,b \in F(X)$ such that $|u| = |p| + |a|$ and $|w| = |p| + |b|$, then $|p| < \lambda\cdot\min(|u|,|w|)$.

Furthermore, as Greedinger's lemma \cite{Greendlinger} applies to presentations with property $\C'(1/6)$, we will only be interested in the case where $\lambda \leqslant 1/6$.
From our definition of small cancellation presentations as given above, a group with property $\C'(1/6)$ is torsion-free hyperbolic.

Moreover, in this note we will only be interested in groups with at least two generators.
Thus, in the remainder of this paper we have $r = |X| \geqslant 2$.

\subsection{Random groups}%
\label{subsec:random-groups}

In this subsection we recall the notion of random groups and random presentations. For more details we refer the reader to the survey \cite{random}.
Notice that we require each relater in a presentation to be cyclically reduced.

In this subsection, we fix a generating set $X$ with cardinality $r=|X| \geqslant 2$.
As stated previously, $\langle X \mid W \rangle$ is a presentation on $m$ relators over $X$ if $W \in F(X)^m$ where each relator $w \in W$ is cyclically reduced.
We write $\mathcal{W}_{m,\ell}$ for the set of all presentations with $m$ relators, each with length at most $\ell$; and $\mathcal{W} = \bigcup_{m=1}^\infty \bigcup_{\ell=1}^\infty \mathcal{W}_{m,\ell}$ for the set of all presentation.

Let $\mathcal{P} \subseteq \mathcal{W}$ be a set of presentations, then we say that a randomly chosen presentation from $\mathcal{W}_{m,\ell}$ belongs to the set $\mathcal{P}$ with probability
\[
    p_{m,\ell}(\mathcal{P})
    =
    \frac{|\mathcal{P}\cap \mathcal{W}_{m,\ell}|}{|\mathcal{W}_{m,\ell}|}.
\]

The main two models of randomness in group theory are the \emph{few relations model} and the \emph{density model}.
We say that a set of presentations $\mathcal{P} \subseteq \mathcal{W}$ is \emph{generic} in the few relation model if, for each $m \geqslant 1$, we have
\[
	\lim_{\ell \to \infty} p_{m,\ell}(\mathcal{P}) = 1.
\]
Furthermore, we say that $\mathcal{P}$ is \emph{strongly generic} if this limit converges exponentially fast.
It was proved in \cite[Lemma~3]{Arzhantseva} that the set of all presentations satisfying the metric small cancellation $\C'(\lambda)$ is strongly generic.

Let some $d$ with $0 \leqslant d \leqslant 1$ be given and let $f_{X,d}(\ell) = (2r-1)^{d\ell}$.
Then, we say that a set of presentations $\mathcal{P}$ is \emph{generic at density $d$} if
\[
    \lim_{\ell \to \infty}
    p_{f_{X,d}(\ell),\ell}(\mathcal{P}) = 1
\]
and we say that $\mathcal{P}$ is \emph{negligible at density $d$} if
\[
    \lim_{\ell \to \infty} 
    p_{f_{X,d}(\ell),\ell}(\mathcal{P}) = 0.
\]

It was proved in \cite[Section~9.B]{Gromov_asymptotic} that, for $0 < \lambda < 1$, the set of all presentations satisfying the metric small cancellation condition $\C'(\lambda)$ is generic at density $d$ if $d < \lambda/2$, and negligible at density $d$ if $d > \lambda/2$.

Using the result in Theorem~\ref{thm:lower} we are able to show that small cancellation is strongly generic with respect to the few relations model, and that small cancellation is generic at densities $d < \lambda/2$.
In particular, from Theorem~\ref{thm:lower} we have the upper bound
\[
	1 - p^\leqslant_\lambda(r,0,\ell,m)
	\leqslant
	8 m^2 r\ell^3 (2r-1)^{-\lambda\ell-1},
\]
where the limit $\lim_{\ell \to \infty}(1 - p^\leqslant_\lambda(r,0,\ell,m)) = 0$ converges exponentially fast for each $r$, $m$ and $\lambda$.
Then, we find that the limit $\lim_{\ell \to \infty} p^\leqslant_\lambda(r,0,\ell,m) = 1$ converges exponentially fast, and thus small cancellation is strongly generic.
Moreover, we find that, for each $0 \leqslant d <1$, we also have the bound
\[
	1 - p^\leqslant_\lambda(r,0,\ell,f_{X,d}(\ell))
	\leqslant
	8 r\ell^3 (2r-1)^{(2d-\lambda)\ell-1}.
\]
Then, we see that $\lim_{\ell\to\infty} (1 - p^\leqslant_\lambda(r,0,\ell,f_{X,d}(d\ell))) = 0$ for each $d < \lambda/2$, and thus small cancellation is generic at density $d$ if $d < \lambda/2$.

Another version of the density model was considered in \cite{colva}, where the authors fix the length and let the number of generators grow. We will call this model the Ashcroft and Roney-Dougal density model. Using Theorem~\ref{thm:lower} we immediately get a statement similar to the positive part of \cite[Section~9.B]{Gromov_asymptotic}.

\begin{prop}
    The set of power-free finite presentations satisfying property $\C'(\lambda)$ is generic in the density model of Ashcroft and Roney-Dougal at densities $d < \lambda/2$.
\end{prop}

\section{Density of Small Cancellation}%
\label{sec:density-of-small-cancellation}

As was mentioned in the previous section, it is well-known that small cancellation is generic, i.e.\ ``almost all'' presentations satisfy metric small cancellation.
However, both \cite[Lemma~3]{Arzhantseva} and \cite[Section~9.B]{Gromov_asymptotic} are purely asymptotic statements and neither informs us of what happens for relatively small parameters.
Thus, in this section we give closed-form formulas for both lower and upper bounds on the probability that a random presentation with given parameters will have small cancellation.

To simplify notation we write $F_r$ to denote a free group of rank $r$, that is, $F_r = F(X)$ where $X = \{ x_1, x_2, \ldots, x_r \}$.
Let $\FR(r,\ell)$ denote the number of freely reduced words of length $\ell$ in $F_r$, then
\[
	\FR(r,\ell)
	=
	2r(2r-1)^{\ell-1}.
\]
Further, let $\CR(r,\ell)$ denote the number of cyclically reduced words of length $\ell$ in $F_r$, then, as was shown by Rivin~\cite[Theorem~1.1]{Rivin},
\[
	\CR(r,\ell)
	=
	(2r-1)^{\ell}
	+ 1
	+ (r-1)\left(1+(-1)^\ell\right).
\]
Moreover, we write $\CR(r,\ell_1, \ell_2)$ to denote the total number of cyclically reduced words of length $\ell$, where $\ell_1 \leqslant \ell \leqslant \ell_2$, in $F_r$.
That is,
\[
	\CR(r,\ell_1,\ell_2)
	=
	\sum_{\ell=\ell_1}^{\ell_2} \CR(r,\ell).
\]

Notice that, if a presentation $\mathcal{P} = \left\langle X \mid W \right\rangle$ does not satisfy small cancellation $\C'(\lambda)$, then it must satisfy at least one of the following two conditions.

\begin{enumerate}
	\item $\NC_\lambda^1$ ---
	there is a relator $w \in W$, and two offsets $d_1,d_2 \in \mathbb{N}$ with $0 \leqslant d_1 < d_2 < |w|$, such that $w' = w_{\ll d_1}$ and $w'' = w_{\ll d_2}$ factor as $w' = xa$, $w'' = yb$ where $a,b,x,y \in F_r$, $x = y^{\pm 1}$ and $|x| \geqslant \lambda |w|$.
	
	\item $\NC_\lambda^2$ ---
	there are two relators $w_1,w_2 \in W$ with cyclic permutations $w_1'$ and $w_2'$, respectively, that factor as $w_1' = xa$ and $w_2' = yb$ where $a,b,x,y \in F_r$, $x = y^{\pm 1}$ and $|x| \geqslant \lambda \cdot \min(|w_1|,|w_2|)$.
\end{enumerate}

\noindent
We write $\NC_\lambda^1(r,\ell)$ to denote the number of length $\ell$ words $w \in F_r$ satisfying property $\NC_\lambda^1$; and $\NC_\lambda^2(r,\ell_1,\ell_2)$ to denote the number of word pairs $w_1,w_2 \in F_r$, each with lengths between $\ell_1$ and $\ell_2$, that satisfying property $\NC_\lambda^2$.
Furthermore, we write $\NC_\lambda^1(r,\ell_1,\ell_2)$ to denote the sum $\sum_{\ell=\ell_1}^{\ell_2} \NC_\lambda^1(r,\ell)$;
and $\NC_\lambda^2(r,\ell)$ to denote $\NC_\lambda^2(r,\ell,\ell)$.

Suppose that we choose a presentation $\mathcal{P} = \left\langle X \mid W \right\rangle$ uniformly at random from the class of presentations with $|X| = r$, $|W| = m$ and $\ell_1\leqslant |w| \leqslant \ell_2$ for each $w \in W$.
Then, we denote the probability of $\mathcal{P}$ having property $\C'(\lambda)$ as $p_\lambda(r,\ell_1,\ell_2,m)$.
In the remainder of this section, we derive lower and upper bounds for this probability.

\subsection{Lower bounds}%
\label{sec:lower-bound}

In the following, we derive a closed-form lower bound $p^\leqslant_\lambda(r,\ell_1,\ell_2,m)$ on the probability of a randomly chosen presentation having small cancellation with the given parameters.

Clearly, we have the lower bound
\begin{equation}%
\label{eq:overall-bound}
    p_\lambda(r,\ell_1,\ell_2,m)
    \geqslant
    1
    -
        m
        \cdot
        \frac{\NC_\lambda^1(r,\ell_1,\ell_2)}{\CR(r,\ell_1,\ell_2)}
    -
        \binom{m}{2}
        \cdot
        \frac{\NC_\lambda^2(r,\ell_1,\ell_2)}{\CR(r,\ell_1,\ell_2)^2}.
\end{equation}

Thus, to find a lower bound on $p_\lambda(r,\ell_1,\ell_2,m)$, we will derive upper bounds on $\NC_\lambda^1(r,\ell)$ and $\NC_\lambda^2(r,\ell_1,\ell_2)$.
In particular, we obtain the bounds given in Lemmas~\ref{lemma:accurate-approximation-of-NC1} and \ref{lemma:accurate-approximation-of-NC2} below.

\begin{lem}%
\label{lemma:accurate-approximation-of-NC1}
	We have the upper bound
	\[
		\NC_\lambda^1(r,\ell)
		\leqslant
		2\ell(\ell-2\left\lceil\lambda\ell\right\rceil - 2)
			\FR(r,\left\lceil \lambda\ell\right\rceil)
			(2r-1)^{\ell-2\left\lceil\lambda\ell\right\rceil}
		+
		\sum_{k=1}^{\left\lceil \lambda\ell \right\rceil}
			\ell
			\CR(r,k)
			(2r-1)^{\ell-\left\lceil \lambda\ell\right\rceil - k}.
	\]
\end{lem}

\begin{proof}
Let $w \in F_r$ be a length $\ell$ cyclically reduced word chosen uniformly at random.
If $w$ satisfies property $\NC_\lambda^1$, then one of the following two cases must apply.
\begin{enumerate}
    \item\label{lemma:accurate-approximation-of-NC1:case 1}
    There is a cyclic permutation $w' = w_{\ll d}$ that factors as both $w' = xa$ and $w' = b_1 y b_2$ where $a,b_1,b_2,x,y \in F_r$ with $x = y^{\pm 1}$, $|x| = \lceil\lambda\ell\rceil$ and $1\leqslant |b_1| \leqslant |x|$.
    
    \item\label{lemma:accurate-approximation-of-NC1:case 2}
    There is a cyclic permutation $w' = w_{\ll d}$ that factors as $w' = xayb$ where $a,b,x,y \in F_r$ with $x = y^{\pm 1}$, $|x| = \lceil\lambda\ell\rceil$ and $|a|,|b| \geqslant 1$.
\end{enumerate}

\noindent
In case~\ref{lemma:accurate-approximation-of-NC1:case 1} it follows that $x = y$ and that $x$ is of the form
\[
    x = (x_1 x_2 \cdots x_k)^p \ x_1 x_2 \cdots x_q
\]
where $k = |b_1|$, $0 \leqslant q<|b_1|$ and the subword $x_1 x_2 \cdots x_k$ is cyclically reduced.

To see this, let $k = |b_1|$ and let $p,q \in \mathbb{N}$ be such that $|x| = p \cdot k + q$ where $0 \leqslant q < k$.
Now suppose $k = |x|$, so that $w'$ factors as $w' = x y b_2$; thus $x \neq y^{-1}$ and $x = x_1 x_2 \cdots x_k$ is cyclically reduced.
Suppose instead that $1 \leqslant k < |x|$; then $x$ and $y$ must factor as $x = x'c$ and $y = cy'$ where $x',y',c \in F_r$ and $c$ is of length $|x|-k$.
Thus, if $x=y^{-1}$, then $c = c^{-1}$ which is not possible as $c\neq \varepsilon$.
Hence, $x = y$ and, since $x$ and $y$ overlap, it follows that
\[
    x = (x_1 x_2 \cdots x_k)^p \ x_1 x_2 \cdots x_q
\]
where the subword $x_1 x_2 \cdots x_k$ is cyclically reduced.

We are now ready to consider the number of words counted in these two cases.
Let us consider case~\ref{lemma:accurate-approximation-of-NC1:case 1}.
Suppose that $k=|b_1|$; then there are $\ell$ choices for the shift $d$, $\CR(r,k)$ choices for the subword $x = (x_1 x_2 \cdots x_k)^p x_1 x_2 \cdots x_q$, and $(2r-1)^{\ell-\lceil \lambda\ell \rceil-k}$ choices for the remaining letters in the word $w$.
Thus, by summing over all such choices for $k$, we obtain
\[
    \sum_{k=1}^{\lceil \lambda\ell \rceil}
        \ell
        \CR(r,k)
        (2r-1)^{\ell-\lceil \lambda\ell \rceil-k}
\]
as an upper bound for the number of counted words.

Now consider case~\ref{lemma:accurate-approximation-of-NC1:case 2}.
There are $\ell$ choices for the offset $d$, $2\cdot\FR(r,\lceil \lambda\ell\rceil)$ choices for the pair $x$ and $y$,  $\ell - 2 \lceil \lambda\ell\rceil - 2$ choices for $|a|$, and $(2r-1)^{\ell-2\lceil \lambda\ell\rceil}$ choices for the remaining letters of the word $w$.
Thus, we obtain
\[
    2\ell(\ell-2\lceil \lambda\ell\rceil-2)
    \FR(r,\lceil \lambda\ell\rceil)
    (2r-1)^{\ell-2\lceil \lambda\ell\rceil}
\]
as an upper bound on the number of such words counted in this case.

Thus, by combining our two previous bounds we obtain our result.
\end{proof}

\begin{cor}%
	\label{cor:rough-bound-on-NC1}
	We have the upper bounds
	\[
		\NC_\lambda^1(r,\ell)
		\leqslant
		4r\ell^2
		(2r-1)^{\ell-\lambda\ell-1}
	\]
	and
	\[
		\NC_\lambda^1(r,\ell_1,\ell_2)
		\leqslant
		4r\ell_2^2 (\ell_2-\ell_1+1)
		(2r-1)^{\ell_2-\lambda\ell_2-1}
	\]
	for each $r \geqslant 2$.
\end{cor}

\begin{proof}
Applying the upper bound
\[
	\CR(r,\ell) \leqslant (2r-1)^\ell + 2r-1
\]
to the inequality given in Lemma~\ref{lemma:accurate-approximation-of-NC1} we obtain
\begin{multline*}
	\NC_\lambda^1(r,\ell)
	\leqslant
	2\ell(\ell-2\left\lceil\lambda\ell\right\rceil-2)
	\cdot
	\FR(r,\lceil\lambda\ell\rceil)
	\cdot
	(2r-1)^{\ell-2\left\lceil\lambda\ell\right\rceil}
	\\
	+
	\sum_{k=1}^{\left\lceil\lambda\ell\right\rceil}
	\ell
	\left[(2r-1)^k + 2r-1\right]
	(2r-1)^{\ell-\left\lceil\lambda\ell\right\rceil -k}.
\end{multline*}
After some rearrangement, we obtain
\[
	\NC_\lambda^1(r,\ell)
	\leqslant
	4r\ell(\ell-2\left\lceil\lambda\ell\right\rceil-2)
	(2r-1)^{\ell-\left\lceil\lambda\ell\right\rceil-1}
	+
	\ell (2r-1)^{\ell-\left\lceil\lambda\ell\right\rceil}
	\sum_{k=1}^{\left\lceil\lambda\ell\right\rceil}
	\left[1 + (2r-1)^{1-k}\right].
\]
From this, we can then obtain  the upper bound
\[
	\NC_\lambda^1(r,\ell)
	\leqslant
	4r\ell(\ell-2\left\lceil\lambda\ell\right\rceil-2)
	(2r-1)^{\ell-\lambda\ell-1}
	+
	\ell (2r-1)^{\ell-\lambda\ell}
	\left(
		\left\lceil\lambda\ell\right\rceil
		+
		2
	\right).
\]
Thus,
\[
	\NC_\lambda^1(r,\ell)
	\leqslant
	\left[
		4r(\ell-2\left\lceil\lambda\ell\right\rceil-2)
		+
		(2r-1)(\left\lceil\lambda\ell\right\rceil+2)
	\right]
	\ell(2r-1)^{\ell-\lambda\ell-1}.
\]
From this upper bound, we can then obtain our bound
\[
	\NC_\lambda^1(r,\ell)
	\leqslant
	4r\ell^2
	(2r-1)^{\ell-\lambda\ell-1}.
\]
Then, using the bound $\sum_{\ell=\ell_1}^{\ell_2} \ell^2 a^\ell \leqslant \ell_2^2 (\ell_2-\ell_1+1) a^{\ell_2}$ for each $a \geqslant 1$ and $1 \leqslant \ell_1 \leqslant \ell_2$, we obtain our bound on $\NC^1_\lambda(r,\ell_1,\ell_2)$.
\end{proof}

\begin{lem}%
\label{lemma:accurate-approximation-of-NC2}
	We have the upper bound
	\[
		\NC_\lambda^2(r,\ell_1,\ell_2)
		\leqslant
		\sum_{j_1=\ell_1}^{\ell_2}
		\sum_{j_2=\ell_1}^{\ell_2}
		2 j_1 j_2
		\FR(r,\left\lceil\lambda\cdot \min(j_1,j_2)\right\rceil)
		(2r-1)^{j_1+j_2-2\left\lceil\lambda\cdot \min(j_1,j_2)\right\rceil}.
	\]
\end{lem}

\begin{proof}
Suppose that we choose two cyclically reduced words $v,w\in F_r$ of lengths $j_1$ and $j_2$, respectively, where $\ell_1 \leqslant j_i \leqslant \ell_2$ for each $j_i$.
Then for the pair of words $v,w$ to satisfy property $\NC_\lambda^2$ there must be cyclic permutations $v' = {v}_{\ll d_1}$ and $w'=w_{\ll d_2}$ that factor as $v' = xa$ and $w' = yb$ where $a,b,x,y\in F_r$, $x=y^{\pm 1}$ and $|x| = \lceil\lambda\cdot\min(j_1,j_2)\rceil$.

Thus, we have $j_1$ possible choices for the offset $d_1$, $j_2$ possible choices for the offset $d_2$, at most $2\cdot\FR(r,\lceil\lambda\cdot\min(j_1,j_2)\rceil)$ possible choices for the pair of words $x$ and $y$, at most $(2r-1)^{j_1 - \lceil\lambda\cdot\min(j_1,j_2)\rceil}$ possible choices for the word $a$, and at most $(2r-1)^{j_2-\lceil\lambda\cdot\min(j_1,j_2)\rceil}$ possible choices for the word $b$.
Hence, we have an upper bound of
\[
    2 j_1 j_2 \FR(r,\lceil\lambda\cdot\min(j_1,j_2)\rceil)
    (2r-1)^{
        j_1+j_2-2\lceil\lambda\cdot\min(j_1,j_2)\rceil
    }
\]
for the number of pairs $v$ and $w$, as before, satisfying property $\NC_\lambda^2$.

Thus, by summing over $j_1$ and $j_2$ from $\ell_1$ to $\ell_2$, we obtain our bound.
\end{proof}

\begin{cor}%
	\label{cor:rough-upper-bound-on-NC2}
	We have the upper bounds
	\[
		\NC_\lambda^2(r,\ell)
		\leqslant
		4r\ell^2(2r-1)^{2\ell-\lambda\ell-1}
	\]
	and
	\[
		\NC_\lambda^2(r,\ell_1,\ell_2)
		\leqslant
		16r\ell_2^2 (\ell_2-\ell_1+1) (2r-1)^{2\ell_2-\lambda\ell_2-1}.
	\]
\end{cor}

\begin{proof}

From the bound in Lemma~\ref{lemma:accurate-approximation-of-NC2} and $\FR(r,\ell) = 2r(2r-1)^{\ell-1}$ we immediately obtain the upper bound
\[
	\NC_\lambda^2(r,\ell)
	\leqslant
	4r\ell^2(2r-1)^{2\ell-\lambda\ell-1}.
\]

To derive our second bound, we rewrite the bound in Lemma~\ref{lemma:accurate-approximation-of-NC2} to obtain
\[
	\NC_\lambda^2(r,\ell_1,\ell_2)
	\leqslant
	2
	\sum_{j_1=\ell_1}^{\ell_2}
	\sum_{j_2=j_1}^{\ell_2}
	2 j_1 j_2
	\FR(r,\left\lceil\lambda j_1\right\rceil)
	(2r-1)^{j_1+j_2-2\left\lceil\lambda j_1\right\rceil}.
\]
We then see that we have the upper estimate
\[
	\NC_\lambda^2(r,\ell_1,\ell_2)
	\leqslant
	8r\ell_2^2(2r-1)^{-1}
	\sum_{j_1=\ell_1}^{\ell_2}
	(2r-1)^{j_1-\lambda j_1}
	\sum_{j_2=j_1}^{\ell_2}
	(2r-1)^{j_2}
	.
\]
Since $(2r-1) > 2$, we have $\sum_{j_2=j_1}^{\ell_2} (2r-1)^{j_2} \leqslant 2(2r-1)^{\ell_2}$, and thus we have
\[
	\NC_\lambda^2(r,\ell_1,\ell_2)
	\leqslant
	16r\ell_2^2(2r-1)^{\ell_2-1}
	\sum_{j_1=\ell_1}^{\ell_2}
	(2r-1)^{j_1-\lambda j_1}
	.
\]
Then, using the bound $\sum_{j_1=\ell_1}^{\ell_2} a^{j_1} \leqslant (\ell_2-\ell_1+1) a^{\ell_2}$ for each $a \geqslant 1$, we have
\[
	\NC_\lambda^2(r,\ell_1,\ell_2)
	\leqslant
	16r\ell_2^2 (\ell_2-\ell_1+1) (2r-1)^{2\ell_2-\lambda\ell_2-1}
\]
as required.
\end{proof}

Using the bounds obtained in this section, we prove Theorem~\ref{thm:lower} as follows.

\begin{proof}[Proof of Theorem~\ref{thm:lower}]
	
Combining the bounds in Lemmas~\ref{lemma:accurate-approximation-of-NC1} and \ref{lemma:accurate-approximation-of-NC2} with the inequality~(\ref{eq:overall-bound}) we obtain a lower bound $p^\leqslant_\lambda(r,\ell_1,\ell_2,m)$ on the probability of small cancellation.
That is, we have $p^\leqslant_\lambda(r,\ell_1,\ell_2,m) \leqslant p_\lambda(r,\ell_1,\ell_2,m)$.

From the upper bound
\[
	\CR(r,\ell_1,\ell_2) \geqslant (2r-1)^{\ell_2}
\]
and the bounds in Corollaries~\ref{cor:rough-bound-on-NC1} and \ref{cor:rough-upper-bound-on-NC2}, we obtain the bound
\begin{multline*}
	1 - p^\leqslant_\lambda(r,\ell_1,\ell_2,m)
	\leqslant
	\frac{m}{(2r-1)^{\ell_2}}\cdot
	4r\ell_2^2 (\ell_2-\ell_1+1)
	(2r-1)^{\ell_2-\lambda\ell_2-1}\\
	+
	\frac{m(m-1)}{2(2r-1)^{2\ell_2}}
	\cdot
	16r\ell_2^2 (\ell_2-\ell_1+1) (2r-1)^{2\ell_2-\lambda\ell_2-1}.
\end{multline*}
Thus, we obtain the upper bound
\[
	1 - p^\leqslant_\lambda(r,\ell_1,\ell_2,m)
	\leqslant
	8 m^2 r\ell_2^2 (\ell_2-\ell_1+1) (2r-1)^{-\lambda\ell_2-1}
\]
as required.
\end{proof}

\subsection{Upper bounds}%
\label{sec:upper-bound}

In this section, we present an upper bound on the probability $p_\lambda(r,\ell,\ell,m)$ in Proposition~\ref{prop:upper-bound-on-probability} below.

\begin{prop}%
\label{prop:upper-bound-on-probability}
	If $\FR(r,\lceil\lambda\ell\rceil) < 2m\ell$, then $p_\lambda(r,\ell,\ell,m)=0$; otherwise
	\[
	\frac{1}
		{\CR(r,\ell)^m}
		\prod_{i=1}^{m}
		\min\Bigg[
			\omega_i(r,\ell,m),\ 
		\beta(r,\ell,m)
		\cdot
		\prod_{k=1}^{\ell_1}
		\min
		\bigg(
			(2r-1)^{\lceil\lambda\ell\rceil},\,
			\alpha_{i,k}(r,\ell,m)
		\bigg)
		\Bigg]
	\]
	is an upper bound for $p_\lambda(r,\ell,\ell,m)$ where
	\begin{align*}
		\omega_i(r,\ell,m)
		&=
			\CR(r,\ell)
			-
			4(i-1)\ell(r-1)(2r-1)^{\ell-\lceil\lambda\ell\rceil - 1},
		\\
		\beta(r,\ell,m)
		&=
			\FR(r,\ell_2),
		\\
		\alpha_{i,1}(r,\ell,m)
		&=
			\FR(r,\lceil\lambda\ell\rceil) - 2(i-1)\ell \mathrm{\ \ and}
		\\
		\alpha_{i,k}(r,\ell,m)
		&=
			\FR(r,\lceil\lambda\ell\rceil)
			-
			2(i-1)\ell
			-
			2\Big(
				(k-2)\lceil\lambda\ell\rceil
				+
				\ell_2
				+
				1
			\Big)
	\end{align*}
	for each $i \geqslant 1$, $k \geqslant 2$ and $\ell = \ell_1\lceil \lambda\ell \rceil + \ell_2$ with $\ell_1,\ell_2\in \mathbb{N}$ and $0 \leqslant \ell_2 < \lceil \lambda\ell \rceil$.
\end{prop}

\begin{proof}
Let $\mathcal{P} = \left\langle X\,|\, W \right\rangle$ be a presentation such that $r = |X|$, $m = |W|$ and each word in the list $W$ is cyclically reduced with length $\ell$.
We write $(w_1,w_2,\ldots,w_m) = W$ for the list of relators, and the length as $\ell=\ell_1\lceil\lambda\ell\rceil+\ell_2$ where $\ell_1,\ell_2 \in \mathbb{N}$ and $0\leqslant \ell_2 < \lceil\lambda\ell\rceil$.
We factor each relator $w_i$ as
\begin{equation}%
\label{eq:w_i-factored}
	w_i
	=
	b_i\,
	a_{i,1}\,
	a_{i,2}\, a_{i,3}\, a_{i,4}\,
	\cdots\, 
	a_{i,\ell_1}
\end{equation}
where each $|a_{i,k}| = \lceil\lambda\ell\rceil$ and $|b_i| = \ell_2$.

If $\mathcal{P}$ satisfies property $\C'(\lambda)$, then each word of the form $(w_i^{\pm 1})_{\ll d}$, with $0 \leqslant d < \ell$, has a distinct length $\lceil \lambda\ell\rceil$ prefix.
Thus, if $\FR(r,\lceil\lambda\ell\rceil) < 2m\ell$, then $p_\lambda(r,\ell,\ell,m)=0$ as there would be no choice for these $2m\ell$ distinct prefixes.
Thus, in the remainder of this proof, we will assume that $\FR(r,\lceil\lambda\ell\rceil) \geqslant 2m\ell$ which also implies that
\begin{equation}%
	\label{eq:upper-bound-on-probability/bound-on-cr}
	\CR(r,\ell) - 4m\ell(r-1)(2r-1)^{\ell-\lceil\lambda\ell\rceil-1} \geqslant 0
\end{equation}
as each such freely reduced word is the prefix of at least \[(2r-2)(2r-1)^{\ell-\lceil \lambda\ell\rceil - 1}\] cyclically reduced words.
Thus, all that remains is to establish our upper bound.

In the remainder of this proof, we place an upper bound on the number of choices for $W$ which result in $\mathcal{P}$ having the small cancellation property $\C'(\lambda)$.
In particular, we will describe a process of choosing relators such that the resulting presentation satisfies property $\C'(\lambda)$.

Suppose that we have already chosen the relators $w_1, w_2, \ldots, w_{i-1}$ in the presentation.
Then, we derive an upper bound on the number of choices for the relator $w_i$ for which the presentation may satisfy property $\C'(\lambda)$.

For $\mathcal{P}$ to satisfy property $\C'(\lambda)$, the length $\lceil\lambda\ell\rceil$ prefix of $w_i$ must be distinct from each length $\lceil\lambda\ell\rceil$ prefix of $(w_j^{\pm 1})_{\ll d}$, where $1 \leqslant j < i$ and $0 \leqslant d < \ell$, which must themselves be pairwise distinct.
Thus, we find that there are $2(i-1)$ prefixes that need to be avoided when choosing the relator.
Moreover, since there are $(2r-2)(2r-1)^{\ell-\lceil\lambda\ell\rceil-1}$ cyclically reduced words corresponding to each avoided prefix, there are at most
\[
    \omega_{i}(r,\ell,m)
    =
    \CR(r,\ell)
    - 4(i-1)\ell(r-1)(2r-1)^{\ell-\lceil\lambda\ell\rceil-1}
\]
choices for the word $w_i$; and
from (\ref{eq:upper-bound-on-probability/bound-on-cr}) we know $w_{i}(r,\ell,m)$ is non-negative.

Now consider the word $w_i$ as written in (\ref{eq:w_i-factored}); we will now place another upper bound on the number of choices for the word $w_i$ by deriving an upper bound on the number of choices for each of its factors.
Firstly, since $w_i$ is cyclically reduced, there are no more than $\beta(r,\ell,m) = \FR(r,\ell_2)$ choices for the factor $b_i$, and no more than $(2r-1)^{\lceil\lambda\ell\rceil}$ choices for each factor of the form $a_{i,j}$.
Moreover, since $a_{i,1}$ must be freely reduced and distinct from each length $\lceil\lambda\ell\rceil$ prefix of some $(w_j^{\pm 1})_{\ll d}$, with $1\leqslant j < i$ and $0\leqslant d < \ell$, we find that there can be at most
\[
    \alpha_{i,1}(r,\ell,m)
    =
    \FR(r,\ell) - 2(i-1)\ell
\]
choices for the factor $a_{i,1}$.
Now suppose that we have made a choice for the factors $b_i a_{i,1} a_{i,2}\cdots a_{i,k-1}$ with $k\geqslant 2$; then the factor $a_{i,k}$ must also avoid each length $\lceil\lambda\ell\rceil$ subword of $(b_i a_{i,1} a_{i,2}\cdots a_{i,k-1})^{\pm 1}$.
Thus, there are at most
\[
    \alpha_{i,k}(r,\ell,m)
    =
	\FR(r,\lceil\lambda\ell\rceil)
	-
	2(i-1)\ell
	-
	2\Big(
		(k-2)\lceil\lambda\ell\rceil
		+
		\ell_2
		+
		1
	\Big)
\]
choices for the factor $a_{i,k}$.

Hence, after making a choice for the words $w_1, w_2, \ldots, w_{i-1}$, we find that there are no more than
\[
    \min\left[
		\omega_i(r,\ell,m),\ 
	\beta(r,\ell,m)
	\cdot
	\prod_{k=1}^{\ell_1}
	\min
	\bigg(
		(2r-1)^{\lceil\lambda\ell\rceil},\,
		\alpha_{i,k}(r,\ell,m)
	\bigg)
	\right]
\]
choices for the word $w_i$.

Thus, by combining our bounds for each $w_i$ we obtain our desired upper bound on the probability $p_\lambda(r,\ell,\ell,m)$.
\end{proof}

\begin{cor}%
	\label{cor:rough-upper-bound}
	If $\FR(r,\lceil \lambda \ell \rceil) \geqslant 2m\ell$, then
	\[
		p_\lambda(r,\ell,\ell,m)
		\leqslant
		\frac{1}{\CR(r,\ell)^{m'}}
		\left(
		\CR(r,\ell)
		- 4 m' \ell(r-1)(2r-1)^{\ell-\lceil\lambda\ell\rceil-1}
		\right)^{m'}
	\]
	where $m' = \lfloor m/2 \rfloor$.
\end{cor}

\begin{proof}
	From Proposition~\ref{prop:upper-bound-on-probability}, we see that, if $\FR(r,\lceil \lambda \ell \rceil) \geqslant 2m\ell$, then
	\[
		p_\lambda(r,\ell,\ell,m)
		\leqslant
		\prod_{i=1}^m
		\frac{\omega_i(r,\ell,m)}{\CR(r,\ell)}
	\]
	where
	\[
		\omega_i(r,\ell,m)
		=
		\CR(r,\ell) - 4(i-1)\ell(r-1)(2r-1)^{\ell-\lceil\lambda\ell\rceil-1}.
	\]
	Then, since $0 \leqslant \omega_i(r,\ell,m) \leqslant \CR(r,\ell)$ where $1\leqslant i\leqslant m$, we see that
	\[
		p_\lambda(r,\ell,\ell,m)
		\leqslant
		\prod_{i=m'+1}^m
		\frac{\omega_i(r,\ell,m)}{\CR(r,\ell)}.
	\]
	Notice that $\omega_i(r,\ell,m) \leqslant \omega_{m'+1}(r,\ell,m)$ for each $i \geqslant m'+1$.
	We see that
	\[
		p_\lambda(r,\ell,\ell,m)
		\leqslant
		\left(\frac{\omega_{m'+1}(r,\ell,m)}{\CR(r,\ell)}\right)^{m'}.
	\]
	That is,
	\[
	p_\lambda(r,\ell,\ell,m)
	\leqslant
	\frac{1}{\CR(r,\ell)^{m'}}
	\left(
	\CR(r,\ell)
	- 4 m' \ell(r-1)(2r-1)^{\ell-\lceil\lambda\ell\rceil-1}
	\right)^{m'}
	\]
	as required.
\end{proof}

From Proposition~\ref{prop:upper-bound-on-probability}, given above, we have an upper bound $p^\geqslant_\lambda(r,\ell,m)$ such that $p^\geqslant_\lambda(r,\ell,m) \geqslant p_\lambda(r,\ell,\ell,m)$.
At the end of the following section, we will see that this upper bound is indeed the one described in Theorem~\ref{thm:upper}.

\section{Finding Limitations on the Parameters}%
\label{sec:probability level sets}

In this section, we derive several conditions for small cancellation to take place with a specified probability.
In particular, we show that, if we wish to have $p_\lambda(r,\ell_1,\ell_2,m)\geqslant p$ for some $p < 1$, then we can do so by either choosing $r$ or $\ell_2$ to be sufficiently large, or, if possible, by choosing $m$ to be sufficiently small.
Moreover, we establish an upper bound on the value of $m$ for small cancellation to occur with a given probability.
This section concludes with a proof of Theorem~\ref{thm:upper}.

\begin{prop}%
	\label{prop:various-bounds}
	If
	\begin{align*}
		\ell_2
		&\geqslant
			e\cdot
			\frac{
				\ln\left(8rm^2\right)
				-
				\ln(1-p)-\ln(2r-1)
			}{
				\lambda e\ln\left(2r-1\right)
				-
				3
			}
		\quad \text{or}\\
		r
		&\geqslant
			\left(
			\frac{8 m^2 \ell_2^2(\ell_2-\ell_1+1)}{1-p}
			\right)^{1/\lambda\ell_2}
	\end{align*}
	then we have $p \leqslant p^\leqslant_\lambda(r,\ell_1,\ell_2,m) \leqslant p_\lambda(r,\ell_1,\ell_2,m)$.
\end{prop}

\begin{proof}

We see that $p \leqslant p^\leqslant_\lambda(r,\ell_1,\ell_2,m)$ if $1-p \geqslant 1-p^\leqslant_\lambda(r,\ell_1,\ell_2,m)$.
Then, from Theorem~\ref{thm:lower}, we have the sufficient condition
\[
	1 - p
	\geqslant 8m^2 r\ell_2^3(2r-1)^{-\lambda\ell_2-1}.
\]
Then, taking the logarithm of both sides, we find that
\[
	\ln(1-p) \geqslant \ln(8m^2r) + 3\ln(\ell_2) + (-\lambda\ell_2-1)\ln(2r-1).
\]
Thus, after rearranging and using the bound $\ln(\ell_2) \leqslant \ell_2/e$, we obtain
\[
		\ell_2
		\geqslant
		e\cdot
		\frac{
			\ln\left(8rm^2\right)
			-
			\ln(1-p)-\ln(2r-1)
		}{
			\lambda e\ln\left(2r-1\right)
			-
			3
		}
\]
as a sufficient condition.

Again, from the bound in Theorem~\ref{thm:lower}, we see that, since $2r-1\geqslant r$, we obtain the sufficient bound
\[
	1-p \geqslant 8 m^2 \ell_2^2(\ell_2-\ell_1+1)r^{-\lambda\ell_2}.
\]
Then, after rearrangement, we obtain the bound
\[
	r
	\geqslant
	\left(
		\frac{8 m^2\ell_2^2(\ell_2-\ell_1+1)}{1-p}
	\right)^{1/\lambda\ell_2}
\]
as required.
\end{proof}

\begin{prop}%
\label{prop:sufficient bound on m}
	If $m$ is such that
	\[
		1
		\leqslant
		m
		\leqslant
		\sqrt{
			\frac{(1-p)(2r-1)^{1+\lambda\ell}}{8r\ell^2}
		},
	\]
	then $p \leqslant p^\leqslant_\lambda(r,\ell,\ell,m) \leqslant p_\lambda(r,\ell,\ell,m)$.
\end{prop}

\begin{proof}
From Theorem~\ref{thm:lower} may derive the sufficient condition
\[
1-p \geqslant 8m^2 r\ell^2(2r-1)^{-\lambda\ell-1}.
\]
Then, after some rearrangement, we obtain the desired result.
\end{proof}

From Proposition~\ref{prop:upper-bound-on-probability}, we may derive the following bound on $m$.

\begin{prop}%
\label{prop:necessary bound on m}
If we have $p^\geqslant_\lambda(r,\ell,m) \geqslant p > 0$, then
\[
    m
    \leqslant
    1+
    2
    \sqrt{
         \frac
            {\ln(1/p)(2r-1)^{1+\lceil\lambda\ell\rceil}}
            {2 \ell(r-1)}
    }.
\]
In particular, the above bound holds if $p^\geqslant_\lambda(r,\ell,m) \geqslant p_\lambda(r,\ell,\ell,m) \geqslant p > 0$.
\end{prop}

\begin{proof}
Firstly, suppose that $\FR(r,\lceil\lambda\ell\rceil) < 2m\ell$; then $p_\lambda(r,\ell,\ell,m) = 0$ by Proposition~\ref{prop:upper-bound-on-probability} and thus our statement holds as there would be no such $p$.
In the remainder of this proof, we suppose that $\FR(r,\lceil\lambda\ell\rceil) \geqslant 2m\ell$ and thus we have the bound in Proposition~\ref{prop:upper-bound-on-probability}.

Then, from Corollary~\ref{cor:rough-upper-bound}, we have
\[
    p_\lambda(r,\ell,\ell,m)
    \leqslant
    \frac{1}{\CR(r,\ell)^{m'}}
    \left(
        \CR(r,\ell)
        - 4 m' \ell(r-1)(2r-1)^{\ell-\lceil\lambda\ell\rceil-1}
    \right)^{m'}
\]
where $m' = \lfloor m/2 \rfloor$.
After some rearrangement, if $p_\lambda(r,\ell,\ell,m)\geqslant p$, then
\[
    \left(
        1
        - m'
        \cdot
        \frac
            {4 \ell(r-1)(2r-1)^{\ell-\lceil\lambda\ell\rceil-1}}
            {\CR(r,\ell)}
    \right)^{m'}
    \geqslant
    p.
\]
Taking the logarithm of both sides we obtain
\[
    m'
    \cdot
    \ln
    \left(
        1
        - m'
        \cdot
        \frac
            {4 \ell(r-1)(2r-1)^{\ell-\lceil\lambda\ell\rceil-1}}
            {\CR(r,\ell)}
    \right)
    \geqslant
    \ln(p).
\]
We can thus apply the Taylor series for $\ln(1-x)$, to obtain
\[
    -m'\sum_{i=1}^{\infty} \frac{1}{i}
    \cdot
    \left(
        m'\cdot
        \frac
            {4 \ell(r-1)(2r-1)^{\ell-\lceil\lambda\ell\rceil-1}}
            {\CR(r,\ell)}
    \right)^i
    \geqslant
    \ln(p)
\]
as a necessary condition.

Hence, we can now see that $m'$ must satisfy
\[
    \left(
        m'
    \right)^2
    \cdot
    \frac
            {4 \ell(r-1)(2r-1)^{\ell-\lceil\lambda\ell\rceil-1}}
            {\CR(r,\ell)}
    \leqslant
    \ln(1/p),
\]
and thus,
\[
    m'
    \leqslant
    \sqrt{
         \frac
            {\CR(r,\ell)\ln(1/p)}
            {4 \ell(r-1)(2r-1)^{\ell-\lceil\lambda\ell\rceil-1}}
    }.
\]
Thus, by taking the upper bound $\CR(r,\ell)\leqslant 2(2r-1)^\ell$, we see that
\[
    m'
    \leqslant
    \sqrt{
         \frac
            {(2r-1)^{1+\lceil\lambda\ell\rceil}\ln(1/p)}
            {2 \ell(r-1)}
    }.
\]
Since $m \leqslant 1+2m'$, we have our result.
\end{proof}

From Proposition~\ref{prop:necessary bound on m}, we may prove Theorem~\ref{thm:upper} as follows.

\begin{proof}[Proof of Theorem~\ref{thm:upper}]

From Proposition~\ref{prop:upper-bound-on-probability} we have $p^\geqslant_\lambda(r,\ell,m) \geqslant p_\lambda(r,\ell,\ell,m)$.
Moreover, after some rearrangement of the bound obtained in Proposition~\ref{prop:necessary bound on m}, we find that
\[
	\ln(1/p^{\geqslant}_\lambda(r, \ell, m))
	\geqslant
	\frac{1}{4} {(m-1)}^2 \ell \frac{2r-2}{{(2r-1)}^{1-\lceil\lambda\ell\rceil}}.
\]
Then, since $2(2r-2) \geqslant 2r-1$ for each $r\geqslant 2$, we see that,
\[
    \ln(1/p^{\geqslant}_\lambda(r, \ell, m))
	\geqslant
	\frac{1}{8} {(m-1)}^2 \ell {(2r-1)}^{-\lceil\lambda\ell\rceil}
\]
 for each $m \geqslant 1$ and $r\geqslant 2$.
\end{proof}

\subsection{Optimal choice of relator length}%
\label{sec:relator-length}

In a way, an optimal choice of length $\ell$ is one for which there exists an integer $k \in \mathbb{N}$ such that $\ell = \lceil k/\lambda \rceil+1$.
For example, if $\lambda = 1/6$, then we would be interested in lengths of the form $\ell = 6\ell_1+1$ as they have the property that \[p_{\lambda}(r,\,6\ell_1+1,\,6\ell_1+1,\,m) \geqslant p_{\lambda}(r,\,6\ell_1+1+\ell_2,\,6\ell_1+1+\ell_2,\,m)\] for each $\ell_2$ with $0 \leqslant \ell_2 < 6$.
This property, as we see below, follows from the definition of small cancellation.

Notice that the length $\ell\geqslant 1$ can be uniquely written  as $\ell = \ell_1/\lambda + \ell_2$ where $\ell_1 \in \mathbb{N}$ and $\ell_2 \in \mathbb{R}$ with $0 < \ell_2 \leqslant 1/\lambda$.
Then, we see that a presentation, $\mathcal{P}=\left\langle X \mid R \right\rangle$, with length $\ell$ relators fails property $\C'(\lambda)$ if and only if either
\begin{enumerate}
	\item there are two words $u,v \in W$ and offsets $d_1,d_2$, with each $0 \leqslant d_i < \ell$, such that $u_{\ll d_1}$ and $v_{\ll d_2}$ share a length $\lceil\lambda\ell\rceil=\ell_1+1$ prefix; or
	\item there is a word $w \in W$ and two offsets $d_1,d_2$, with $0\leqslant d_1 < d_2 < \ell$, such that $w_{\ll d_1}$ and $w_{\ll d_2}$ share a length $\lceil\lambda\ell\rceil=\ell_1+1$ prefix.
\end{enumerate}

\noindent
Thus, we see that increasing $\ell_2$ within the range $0 < \ell_2 \leqslant 1/\lambda$ can only increase the probability of $W$ containing such a choice of words and thus decrease the probability of small cancellation.
Hence, with $\ell$ in the range $\lceil k/\lambda\rceil + 1 \leqslant \ell < \lceil (k+1)/\lambda\rceil+1$, the probability, $p_\lambda(r,\ell,\ell,m)$, of small cancellation is maximal at $\ell = \lceil k/\lambda\rceil+1$.

\appendix
\section{Experimental Results}%
\label{sec:experimental-results}

In this appendix we compare our lower and upper bounds, from Section~\ref{sec:lower-bound} and Proposition~\ref{prop:upper-bound-on-probability} respectively, with estimates of $p_\lambda(r,\ell,\ell,m)$ obtained from computational experiment.
The code used to create this section is provided at \cite{githubcode}.
In particular, we present several heatplots which show how our bounds on $p_\lambda(r,\ell,\ell,m)$ compare as we vary the values of $r$, $\ell$ and $m$.
Each data-point in each heatplot was obtained from a data sample consisting of at least 35\,000 randomly chosen presentations.
Within this appendix, unless otherwise specified, $\lambda=1/6$.

In Figure~\ref{fig:l-vs-m heatmap} we fix the number of generators, $r$, to 20, and compare the probability of small cancellation, $p_\lambda(20,\ell,\ell,m)$, as we vary the number of relators, $m$, and the length of such relators, $\ell$.
Counterintuitively, it appears that the probability of small cancellation is not monotone non-decreasing with respect to the relator length $\ell$.
In fact, the probability appears to be decreasing within ranges of length $6=1/\lambda$.
A similar phenomenon appears again in Figure~\ref{fig:l-vs-r heatmap}, in which the number of relators, $m$, is fixed to $10$ and the probability $p_\lambda(r,\ell,\ell,10)$ is compared as $r$ and $\ell$ are varied.
Moreover, we see that, if we instead set $\lambda = 1/100$, as in Figure~\ref{fig:l-vs-r heatmap100}, then we obtain the same pattern where the probability decreases within ranges of size $100=1/\lambda$.
The reason behind this pattern is explained in Section~\ref{sec:relator-length}.

Finally, in Figure~\ref{fig:r-vs-m heatmap}, we fix the relator length, $\ell$, to 20, and compare the probabilities of small cancellation, $p_\lambda(r,20,20,m)$, as we vary the number of generators, $r$, and relators, $m$.

\begin{figure}[!htp]
\centering

\subfloat[Lower bound from Section~\ref{sec:lower-bound}.]{
	\centering
	\includegraphics[width=.48\linewidth]{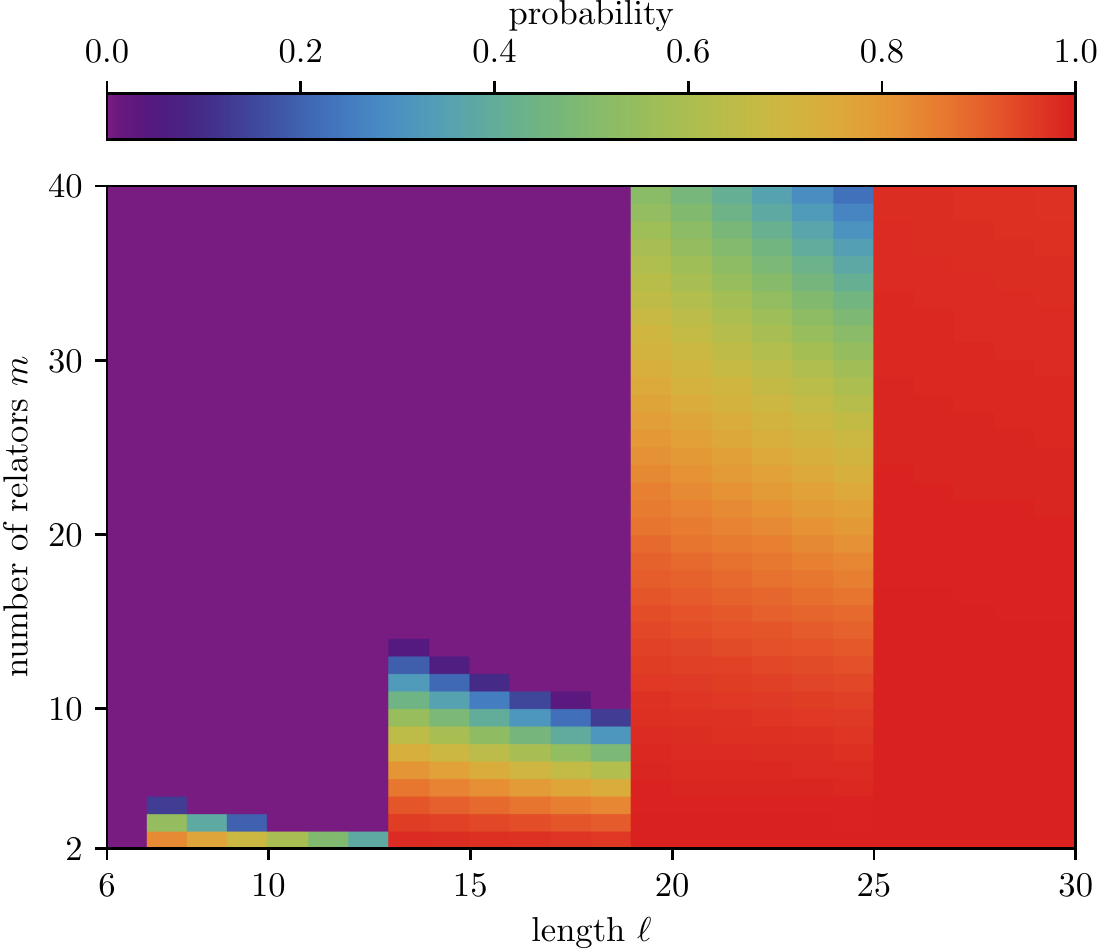}
}%
\hfill%
\subfloat[Upper bound from Proposition~\ref{prop:upper-bound-on-probability}.]{
	\centering
	\includegraphics[width=.48\linewidth]{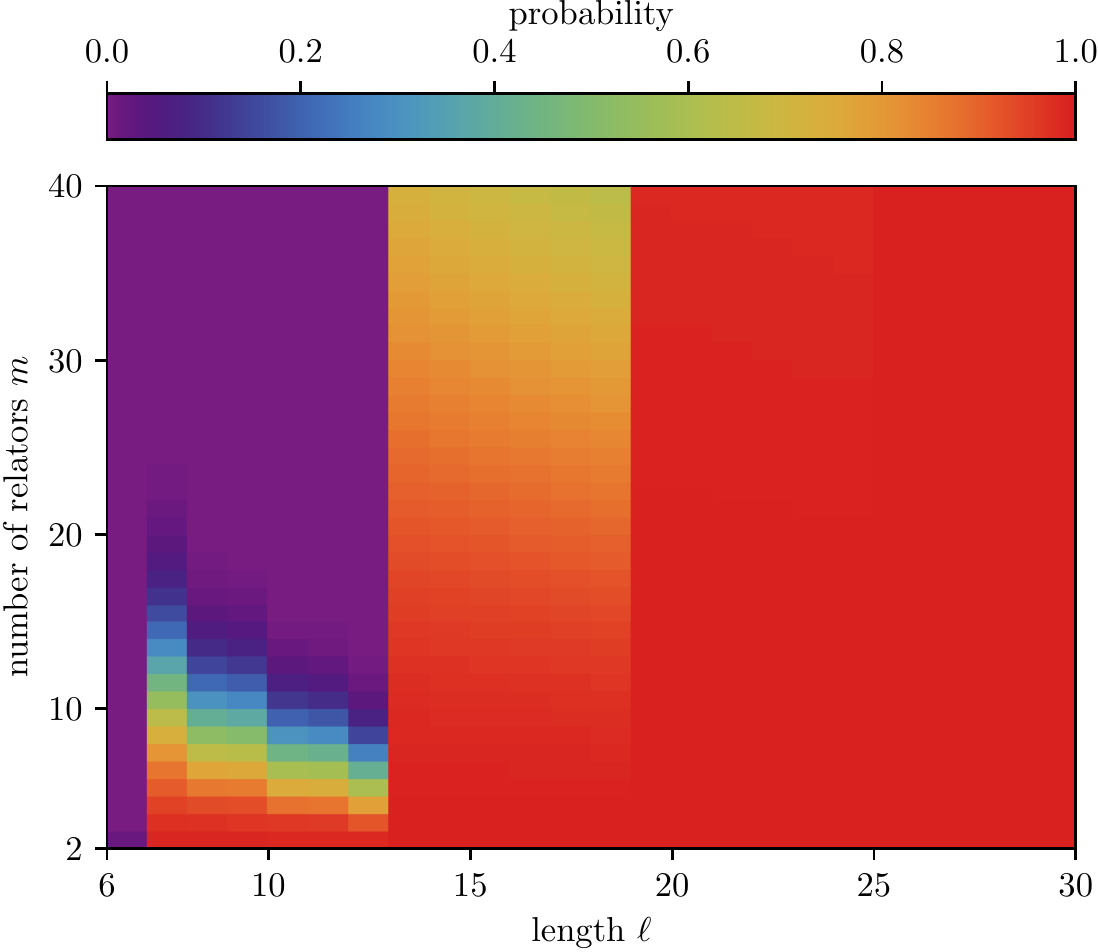}
}

\vspace{1.5em}

\subfloat[Experimental approximation.]{
	\centering
	\includegraphics[width=.48\linewidth]{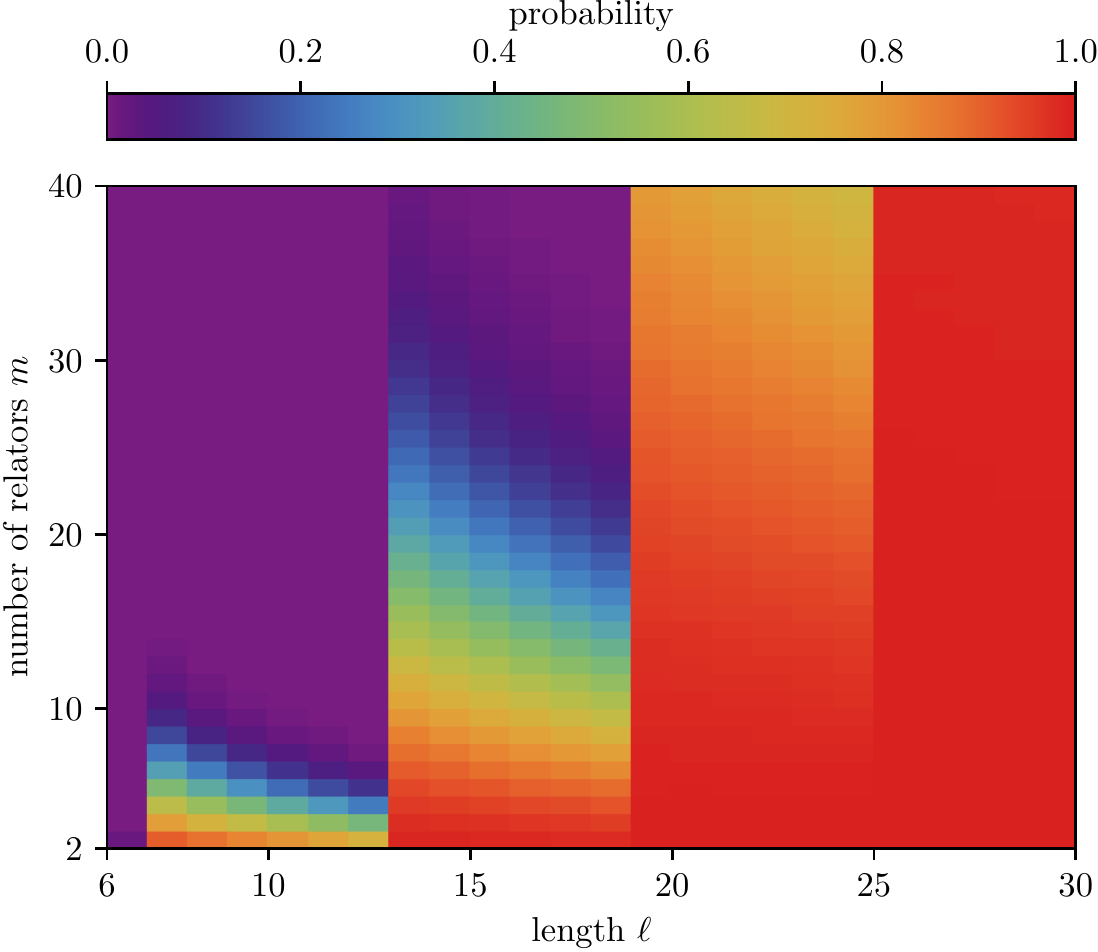}
}

\caption{Heatmaps giving upper and lower bounds, and an experimental approximation of $p_{\lambda}(20,\ell,\ell,m)$ as $\ell$ and $m$ are varied, with $r$ fixed to be $20$.}%
\label{fig:l-vs-m heatmap}
\end{figure}

\begin{figure}[!htp]
\centering

\subfloat[Lower bound from Section~\ref{sec:lower-bound}.]{
	\centering
	\includegraphics[width=.48\linewidth]{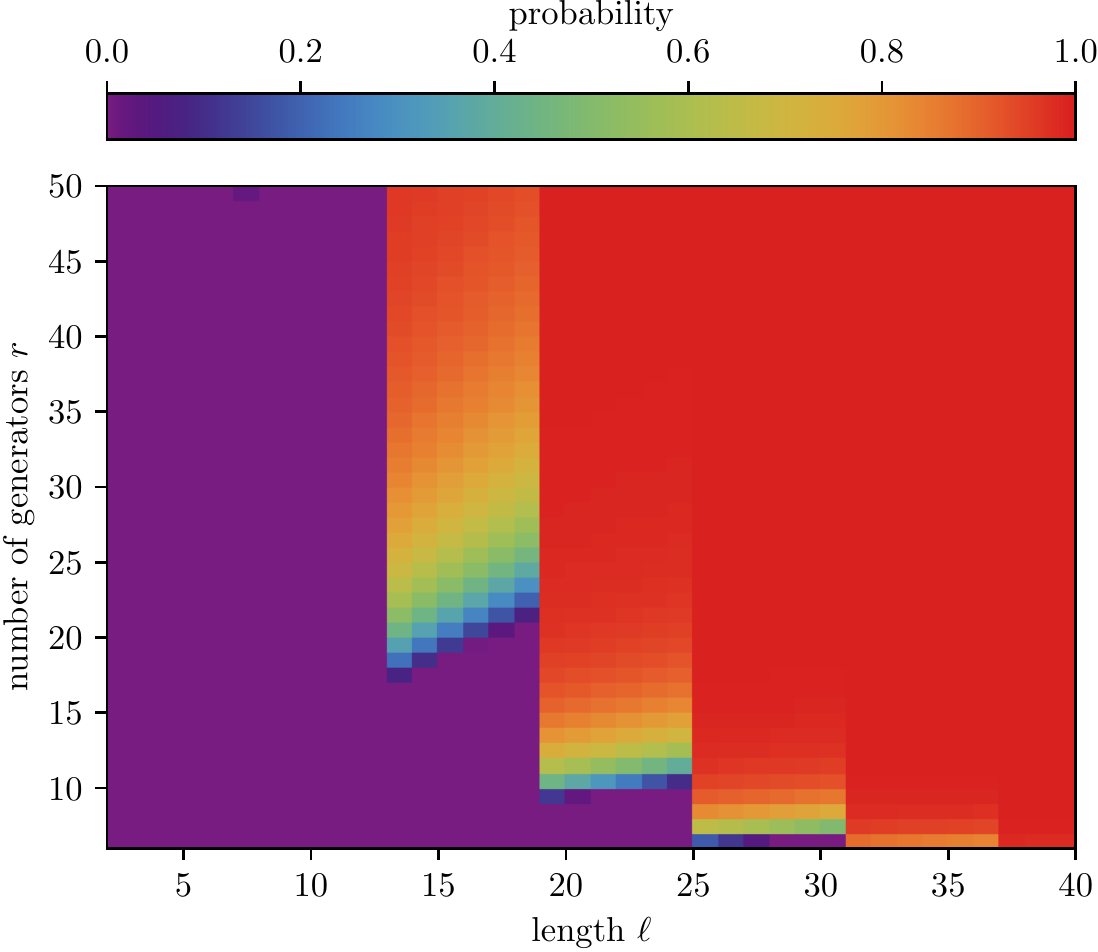}
}%
\hfill%
\subfloat[Upper bound from Proposition~\ref{prop:upper-bound-on-probability}.]{
	\centering
	\includegraphics[width=.48\linewidth]{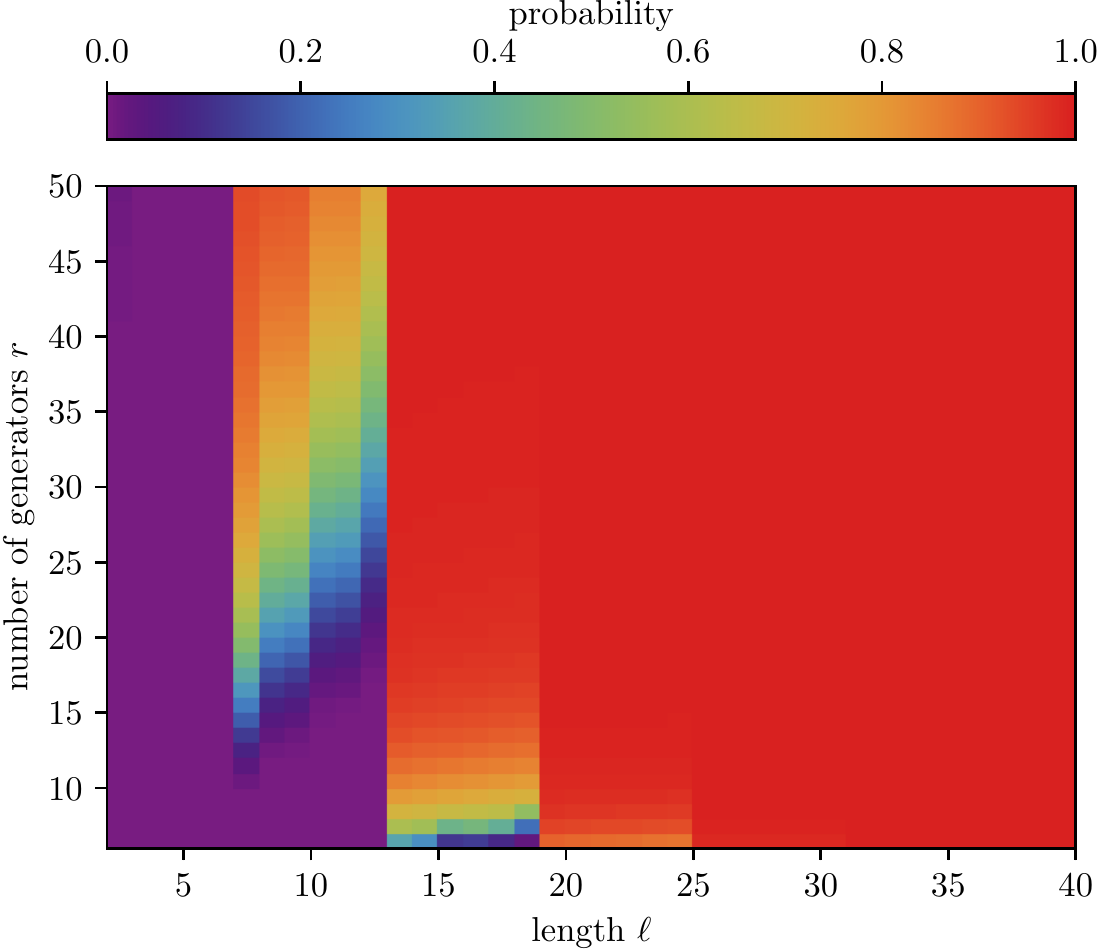}
}

\vspace{1.5em}
\subfloat[Experimental approximation.]{
	\centering
	\includegraphics[width=.48\linewidth]{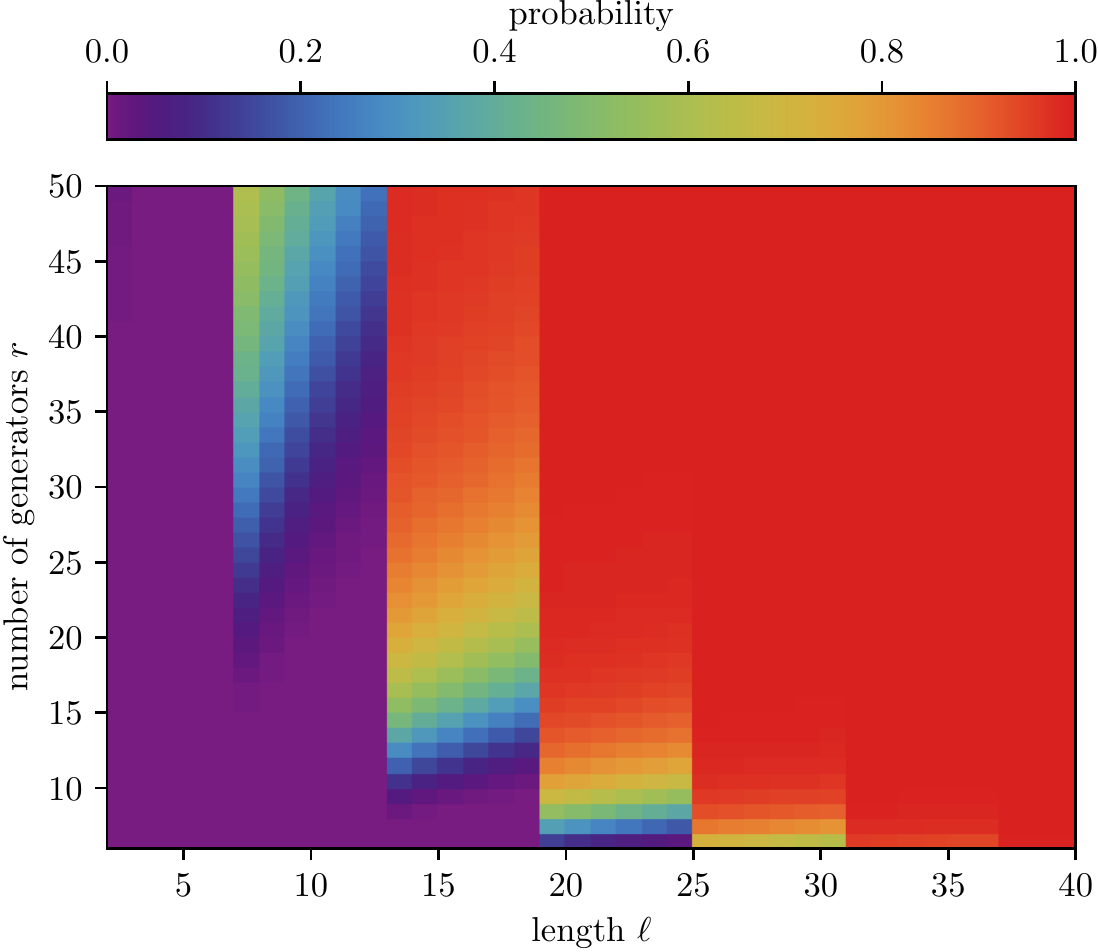}
}

\caption{Heatmaps giving upper and lower bounds, and an experimental approximation of $p_{\lambda}(r,\ell,\ell,10)$ as $\ell$ and $r$ are varied, with $m$ fixed to be $10$.}%
\label{fig:l-vs-r heatmap}
\end{figure}

\begin{figure}[!htp]
\centering

\subfloat[Lower bound from Section~\ref{sec:lower-bound}.]{
	\centering
	\includegraphics[width=.48\linewidth]{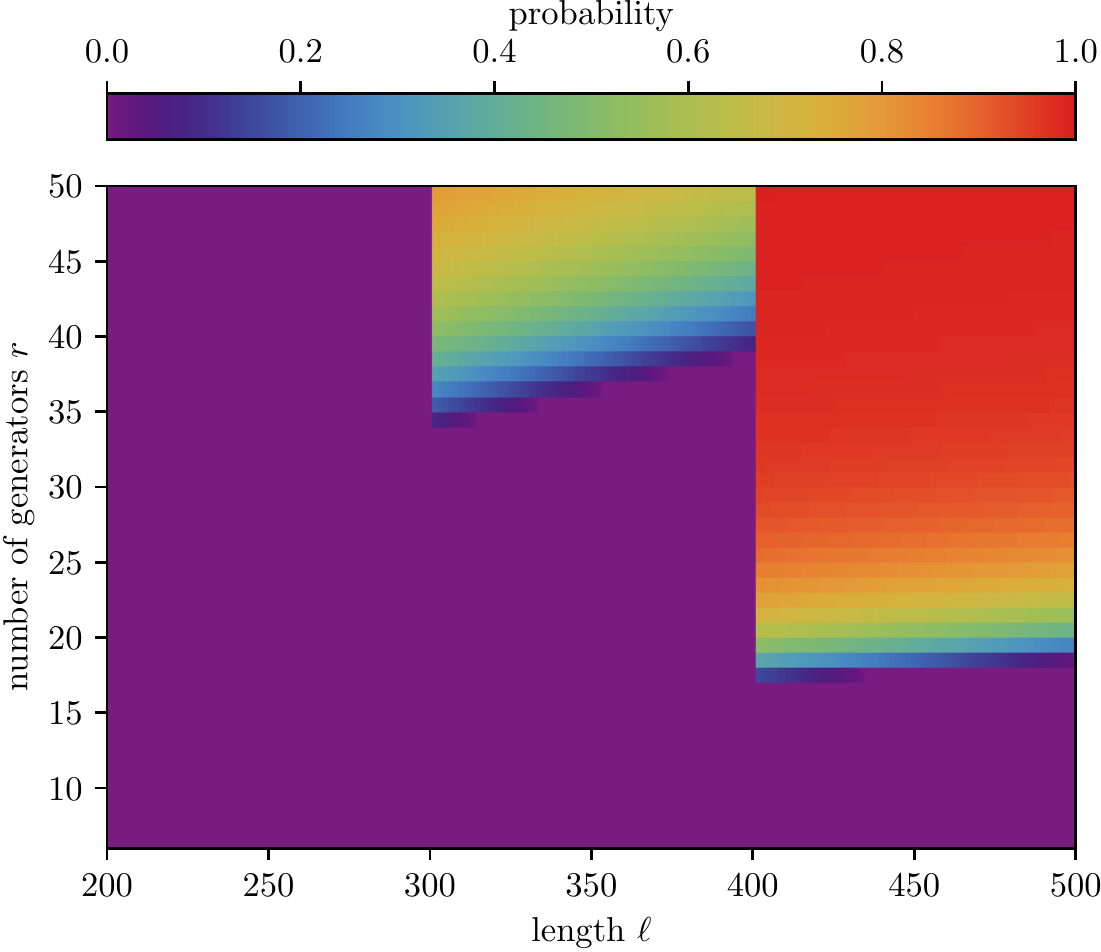}
}%
\hfill%
\subfloat[Upper bound from Proposition~\ref{prop:upper-bound-on-probability}.]{
	\centering
	\includegraphics[width=.48\linewidth]{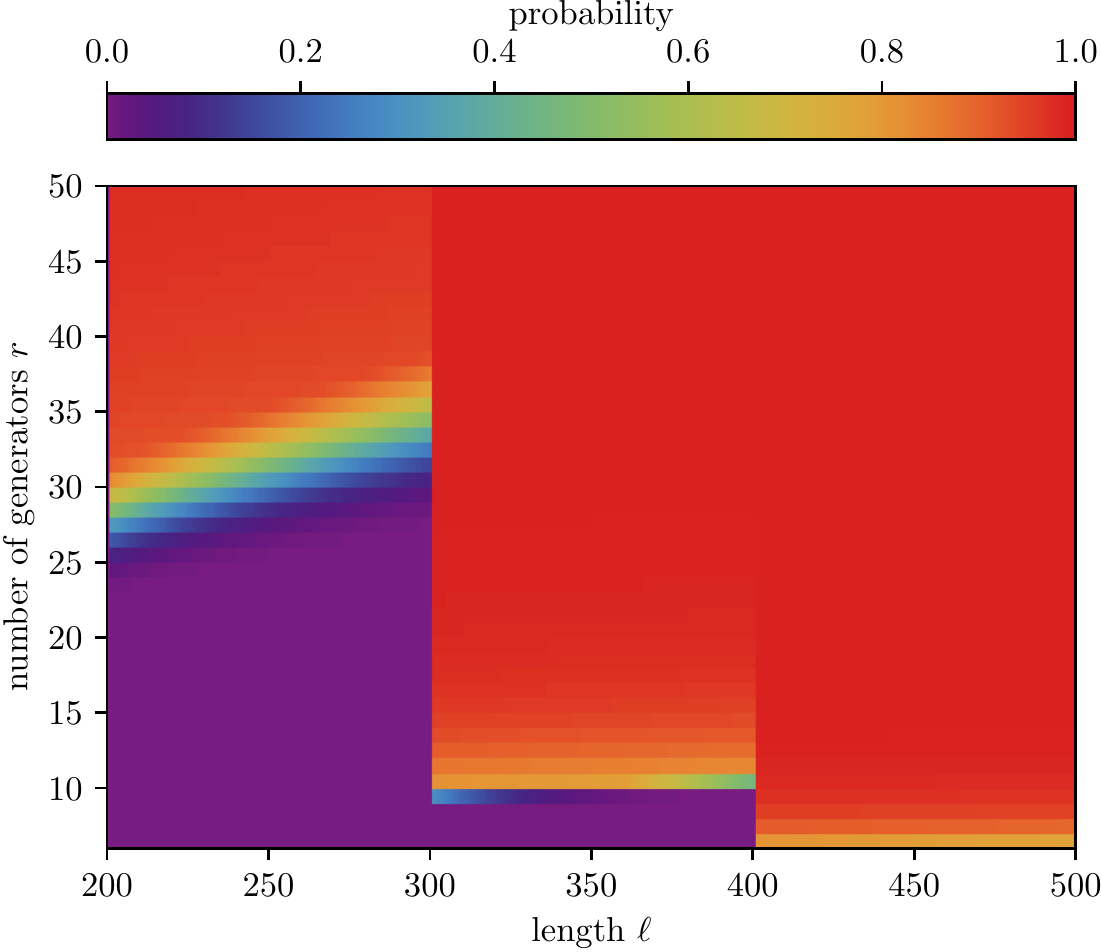}
}

\vspace{1.5em}

\subfloat[Experimental approximation.]{
	\centering
	\includegraphics[width=.48\linewidth]{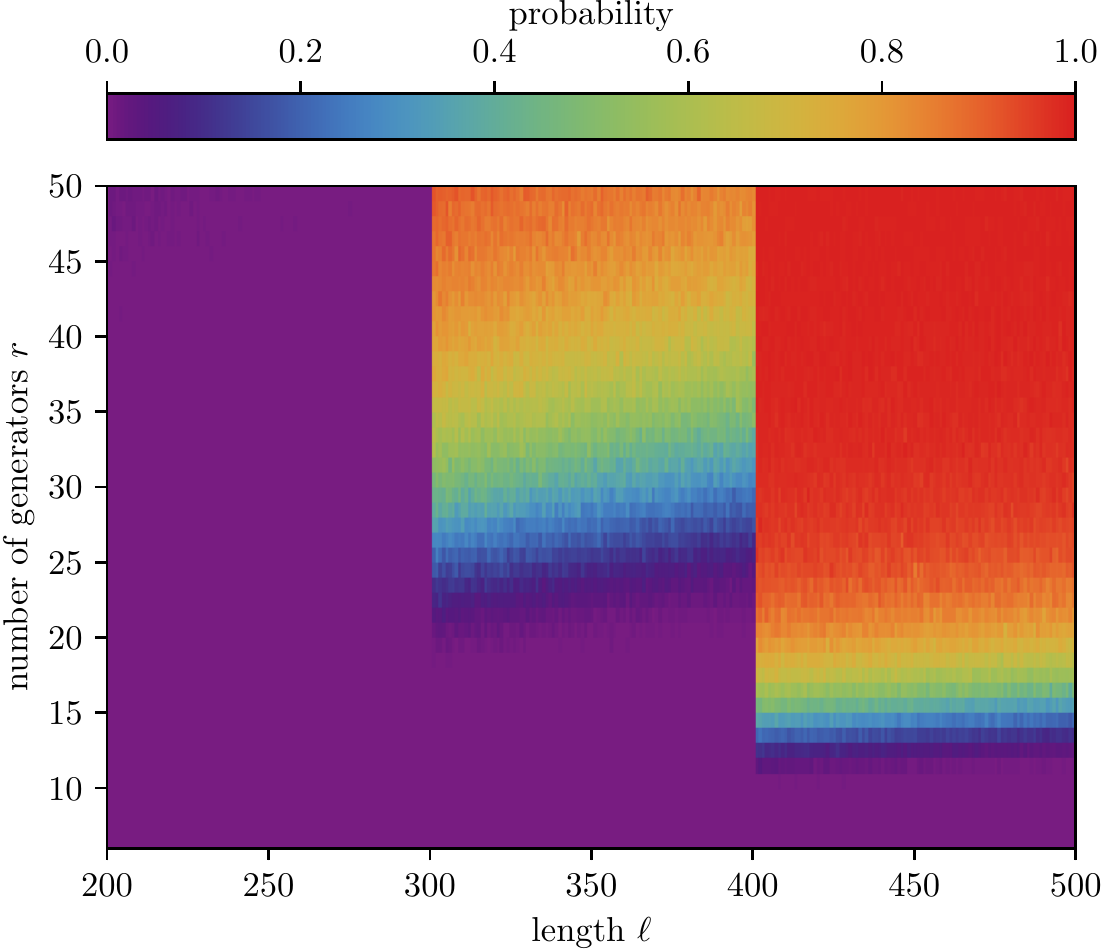}
}

\caption{Heatmaps giving upper and lower bounds, and an experimental approximation of $p_{1/100}(r,\ell,\ell,10)$ as $\ell$ and $r$ are varied, with $m$ fixed to be $10$ and $\lambda=1/100$.}%
\label{fig:l-vs-r heatmap100}
\end{figure}

\begin{figure}[!htp]
\centering

\subfloat[Lower bound from Section~\ref{sec:lower-bound}.]{
	\centering
	\includegraphics[width=0.48\linewidth]{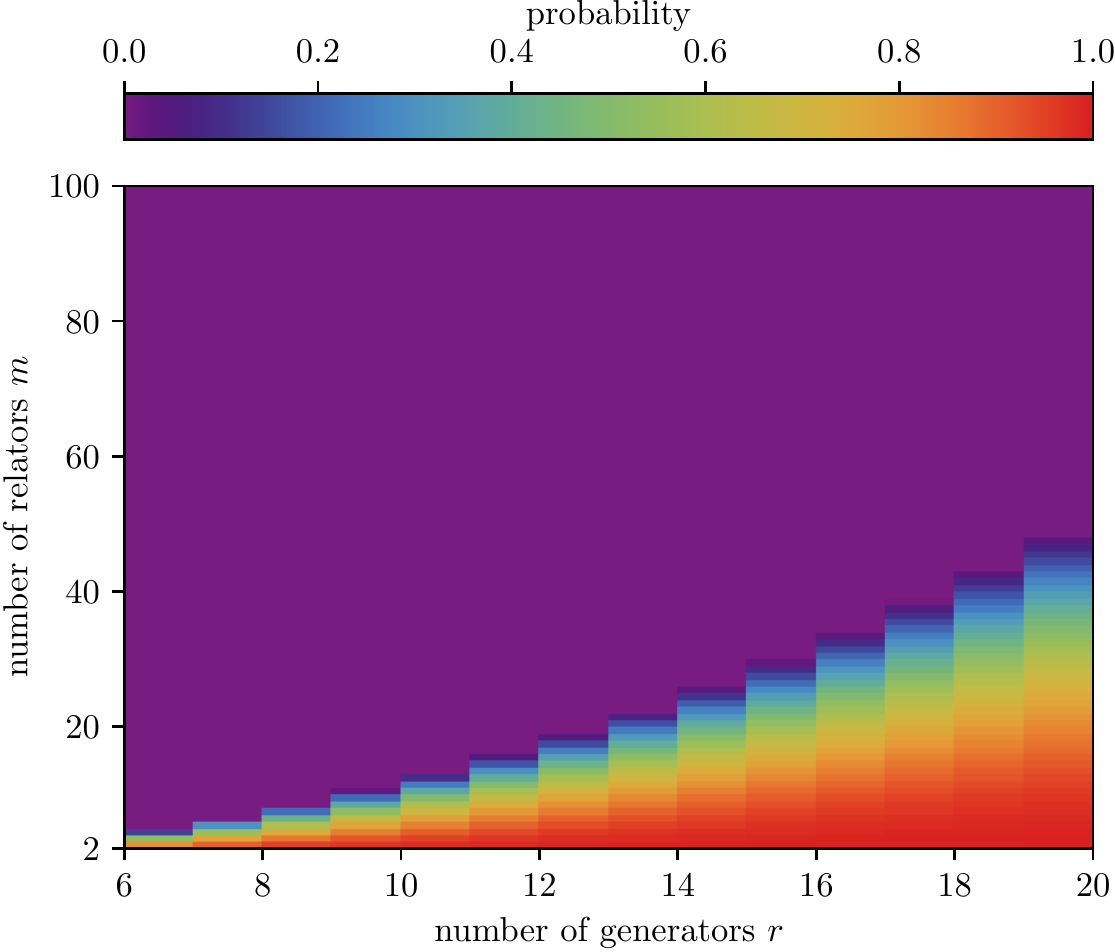}
}%
\hfill%
\subfloat[Upper bound from Proposition~\ref{prop:upper-bound-on-probability}.]{
	\centering
	\includegraphics[width=0.48\linewidth]{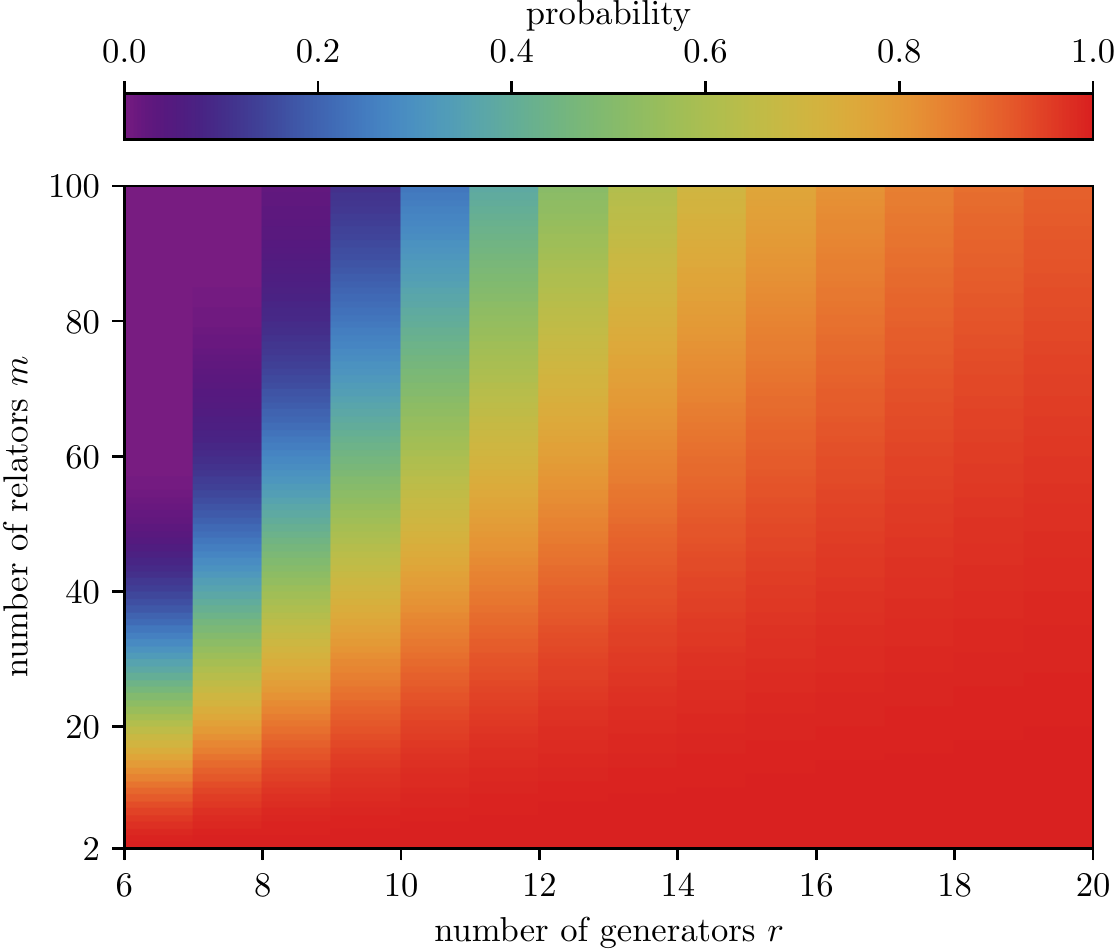}
}

\vspace{1.5em}

\subfloat[Experimental approximation.]{
	\centering
	\includegraphics[width=0.48\linewidth]{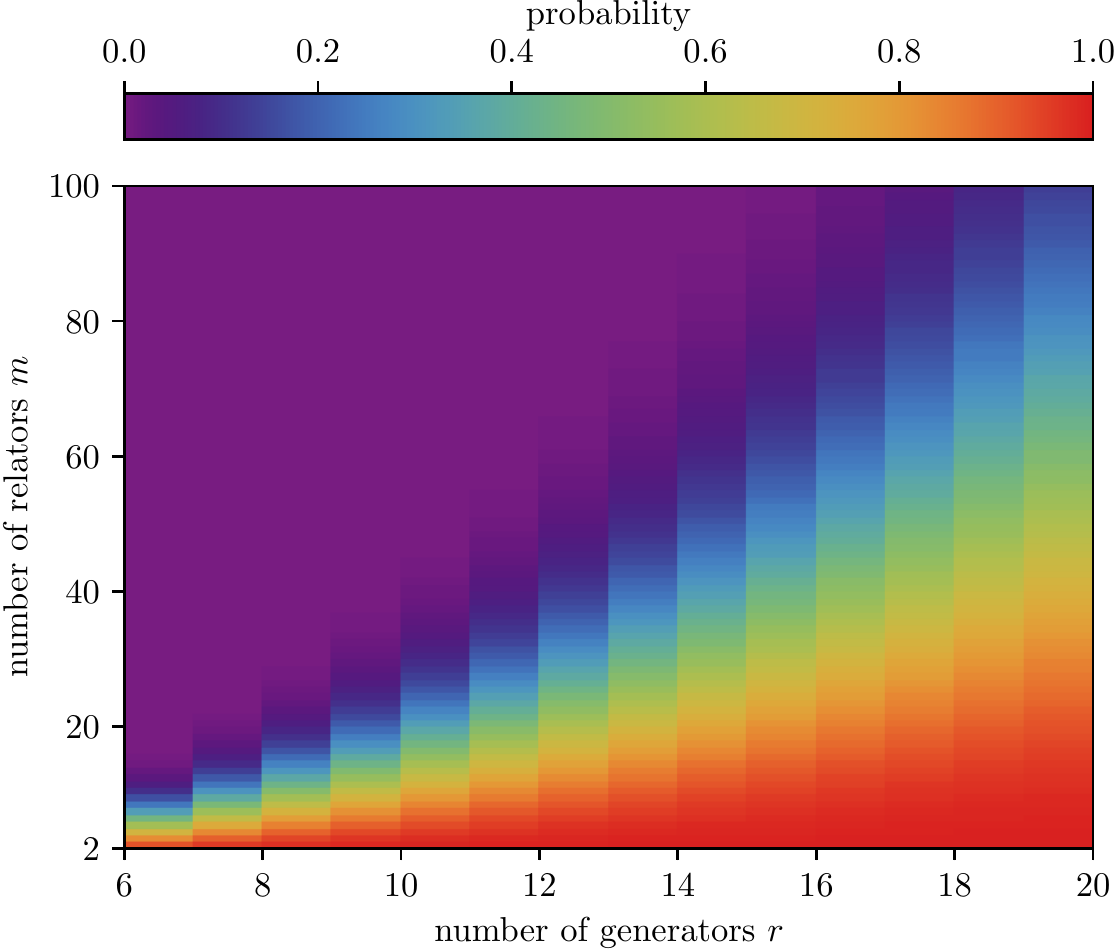}
}

\caption{Heatmaps giving upper and lower bounds, and an experimental approximation of $p_{\lambda}(r,20,20,m)$ as $r$ and $m$ are varied, with $\ell$ fixed to be $20$.}%
\label{fig:r-vs-m heatmap}
\end{figure}

\clearpage

\section*{Acknowledgements}

A large proportion of the theoretical work presented in this paper was a part of the second named author's master thesis at Charles University in Prague, Czech Republic, which was supervised by Pavel P\v r\'\i hoda.
The first named author acknowledges support from an Australian Government Research Training Program Scholarship.
The authors also acknowledge support from Australian Research Council grant DP160100486.

\bibliography{refs}
\bibliographystyle{plain}

\end{document}